\documentclass[abstract=no]{amsart} 
\usepackage[margin=3.8cm]{geometry}

\usepackage{amsmath}
\usepackage{bbm}
\usepackage{latexsym}
\usepackage[english]{babel}
\usepackage{amsfonts}
\usepackage[T1]{fontenc}
\usepackage[utf8]{inputenc}
\usepackage{amsthm}
\usepackage{amssymb}
\usepackage{dsfont}
\usepackage{version}
\usepackage{caption}
\usepackage[backend=bibtex,style=alphabetic,hyperref=true,maxnames=5]{biblatex}
\usepackage[psamsfonts]{eucal}
\usepackage{stmaryrd}
\usepackage{url}
\usepackage{hyperref}
\usepackage{color}
\usepackage{blindtext}
\usepackage{tikz}
\usetikzlibrary{arrows,matrix,positioning}
\tikzstyle{dmatrix}=[matrix of math nodes,row sep=2.5em, column sep=2.5em,
text height=1.5ex, text depth=0.25ex] 
\usepackage{mathtools}
\usepackage{microtype}
\usepackage{graphicx}
\usepackage{mathabx}
\usepackage{lipsum}
\usepackage{float}
\usepackage{csquotes}
\usepackage{empheq}
\usepackage{enumitem}
\usepackage{mathrsfs}
\usepackage{mathtools}
\DeclareSymbolFont{epsilon}{OML}{ntxmi}{m}{it}
\DeclareMathSymbol{\epsilon}{\mathord}{epsilon}{"0F}

\theoremstyle{plain}
\newtheorem{theorem}{Theorem}[section]
\newtheorem{lemma}[theorem]{Lemma}
\newtheorem{prop}[theorem]{Proposition}
\newtheorem{cor}[theorem]{Corollary}

\newtheorem{prop/Def}[theorem]{Propsition/Definition}
\newtheorem{theorem/Def}[theorem]{Theorem/Definition}
\theoremstyle{definition}

\newtheorem{Def}[theorem]{Definition}
\newtheorem{rem}[theorem]{Remark}
\newtheorem{exa}[theorem]{Example}

\def \Q {{\mathbb Q}}
\def \R {{\mathbb R}}

\def \Z {{\mathbb Z}}

\def \P {{\mathbb P}}
\def \X {{ \mathcal{X}}}

\def \T {{\mathbb T}}
\def \div {{\operatorname{div}}}

\def \Im {{ \operatorname{Im}}}

\def \conv {{ \operatorname{conv}}}

\def \codim {{\operatorname{codim}}}
\newcommand{\on}[1]{\operatorname{#1}}
\newcommand{\ca}[1]{{\mathcal{#1}}}
\def \div {{ \operatorname{div}}}

\def \cone {{ \operatorname{cone}}}

\def \sp {{\operatorname{sp}}}

\def \rec {{\operatorname{rec}}}

\def \CH {{\operatorname{CH}}}

\def \sp {{\operatorname{sp}}}
\def \Div {{\operatorname{Div}}}

\renewbibmacro{in:}{}

\usepackage{tgpagella}
\usepackage{tgheros}
\usepackage{eulervm}
\usepackage{tikz,tikz-cd}
\usepackage{etoolbox}
\usepackage{subfigure}
\usetikzlibrary{arrows,matrix,positioning}
\tikzstyle{dmatrix}=[matrix of math nodes,row sep=2.5em, column sep=2.5em,
text height=1.5ex, text depth=0.25ex] 
\usepackage{color}

\numberwithin{equation}{section}

\bibliography{bibliography_update}

\title{On Chow groups of toric schemes over a DVR}
\author{Ana Mar\'ia Botero}
\date{}
\thanks{The author was supported by the collaborative research 
center SFB 1085 \emph{Higher Invariants - Interactions between
  Arithmetic Geometry and Global Analysis} funded by the Deutsche
Forschungsgemeinschaft. }
\subjclass{14M25; 14C15; 14C17; 05E14}
\keywords{Chow groups and rings, toric schemes, intersection theory, polyhedral complexes} 
 
\begin{document}

\maketitle

\begin{abstract}
We give an explicit combinatorial presentation of the Chow groups of a toric scheme over a DVR.  As an application, we compute the Chow groups of several toric schemes over a DVR and of their special fibers. 
\end{abstract}

\tableofcontents
\section{Introduction}
Toric varieties over a field have been studied for over 50 years. Their geometric and topological properties are governed by the convex geometry of rational polyhedral fans. This makes toric varieties an excellent testing ground for conjectures in algebraic geometry. Main references for toric varieties over a field are \cite{CLS}, \cite{ful}, \cite{KKMD} and \cite{oda}. Although in these references the authors normally work over an algebraically closed field, most of the theory extends to arbitrary fields without any obstruction. 

In this article we study the geometry of toric schemes over a discrete valuation ring (DVR).  Our main result is a combinatorial presentation of its Chow (homology) groups.  

Our setting is as follows. Let $K$ be a field equipped with a non-trivial discrete valuation. We denote by $R$ the corresponding valuation ring and by $\kappa$ the residue field.
Let $S = \operatorname{Spec}\left(R\right)$ and denote by $\eta$ and by $s$ the generic and the special point of $S$, respectively. For a scheme $\mathcal{X}$ over $S$, we set $\mathcal{X}_{\eta} \coloneqq \mathcal{X} \otimes_S \operatorname{Spec}(K)$ and $\mathcal{X}_{s} \coloneqq \mathcal{X} \otimes_S \operatorname{Spec}(\kappa)$ for the generic and the special fiber, respectively. 

Assume that $N$ is the lattice of one-parameter subgroups of a torus $\mathbb{T}_K$ of diemsnion $n$. As is usual in toric geometry, $M = N^{\vee}$ denotes its dual lattice. Let $\widetilde{N} \coloneqq N \oplus \Z$, $\widetilde{M} \coloneqq M \oplus \Z$ and set $\mathbb{T}_S = \on{Spec}\left(R[M]\right) \simeq \mathbb{G}_{m,S}^n$ a split torus over $S$. We consider toric schemes over $S$ (see Definition~\ref{def:toric-scheme}). 

Toric schemes over $S$ were first considered in \cite[section IV.3]{KKMD}, mainly motivated by compactification and degeneration problems.  A similar motivation appears later in \cite{smirnov}. More recently,  they have been considered in connection to tropical geometry (see e.g.~\cite{katz}). They have also played an important role for arithmetic purposes. In \cite{BPS} they are used to investigate the Arakelov geometry of toric varieties.  

Toric schemes over $S$ are described in terms of fans in $\widetilde{N}_{\R}=  N_{\mathbb{R}}\oplus\mathbb{R}$.
If we restrict to proper schemes, these can also be characterized by complete strongly convex rational (SCR)  polyhedral complexes in $N_{\R}$ (see Definition \ref{def:SCR} and Theorem \ref{th:corr-comp}). 

As in the case of toric varieties over a field, many geometric properties of a toric scheme $\X$ over $S$ can be read from the underlying combinatorial structure. For example, if $\X$ is regular and proper, and if $\Pi$ is the associated smooth complete SCR polyhedral complex (see Defintion \ref{def:smooth}), then the generic fiber $\X_{\eta}$ is a toric variety over the field $K$ with associated fan the recession fan of $\Pi$.  The (reduced) special fiber $\X_s$ is a union of toric varieties indexed by the polyhedra in $\Pi$, glued in a way that is compatible with the rational polyhedral structure (see Section~\ref{subsec:correspondence} for details).  Moreover, a cone $\sigma$ in the recession fan of $\Pi$ induces a horizontal cycle $V(\sigma)$, while a bounded polyhedron $\Lambda \in \Pi$ gives a vertical cycle $V(\Lambda)$. 
\subsection{Statement of the main results}

Let $\X$ be a proper regular scheme over $S$. We consider the rational Chow (homology) groups $\CH_k(\X/S)$(see Definition \ref{def:chow}). Note that we assume our Chow groups are tensored with $\Q$ and that we work with the notion of $S$-\emph{absolute} dimension. This coincides with the notion of \emph{relative} dimension given in \cite[Section~20.1]{fulint} plus 1. We choose to work with this absolute notion because it coincides with the one discussed in~\cite{BGS} and \cite{bot-ara}.  A reason for the name ``$S$-absolute dimension'' is given in Remark \ref{rem:dimension}. 

Let now $\X = \X_{\Pi}$ be the proper toric scheme over $S$ associated to a complete rsmooth SCR polyhedral complex $\Pi$ in $N_{\R}$ and set $\Sigma = \rec(\Pi)$ for the corresponding recession fan.

 Our first main result states that $\CH_k(\X/S)$ is generated by invariant cycles and gives a set of generators.  The following is Theorem \ref{th:inv-cycles}.
 
 \begin{theorem}\label{th:intro0}
Let $k\geq 0$ be an integer. Then the Chow group $\CH_k(\X/S)$ is generated by torus invariant horizontal and vertical cycles of $S$-absolute dimension $k$. A collection of generators is given by the family of horizontal cycles $\left[V(\sigma)\right]$ for $\sigma$ a $(n-k+1)$-dimensional cone in $\Sigma$ and of vertical cycles $\left[V(\Lambda)\right]$ for~$\Lambda$ an $(n-k)$-dimensional bounded polyhedron in $\Pi$. 
\end{theorem}

Our second main result gives a combinatorial presentation of the Chow groups $\CH_k(\X/S)$. In particular, it describes all of the relations between the above set of generators.  Before stating the result, we introduce some notation. For a cone $\tau \in \Sigma$ let $N(\tau)$ denote the quotient lattice $N/\left(N\cap \R\tau\right)$. Also, set $M(\tau) \coloneqq N(\tau)^{\vee}$. Similarly, for a polyhedron $\Lambda \in \Pi$ we let $\widetilde{N}(\Lambda) \coloneqq \widetilde{N}/\left(\widetilde{N}\cap \R \operatorname{cone}(\Lambda)\right)$ and $\widetilde{M}(\tau) \coloneqq \widetilde{N}(\tau)^{\vee}$.

The following is Theorem \ref{th:exact}.

\begin{theorem}\label{th:intro1}
Let $k\geq 0$ be an integer. Then the Chow group $\CH_k(\X/S)$ fits into an exact sequence 
\begin{align*}
\bigoplus_{\tau \in \Sigma(n-k)}\left(M(\tau)_{\Q} \oplus \Q\right)\oplus \bigoplus_{\Lambda \in \Pi(n-k-1)}\widetilde{M}(\Lambda)_{\Q} \to \Q^{|\Sigma(n-k+1)|} \oplus \Q^{|\Pi(n-k)|} \to \CH_{k}(\X/S) \to 0.
\end{align*}
\end{theorem}
The maps of the above sequence are defined from a localization sequence and are explicitly described in Section \ref{sec:chow-hom}.

 In order to relate these result to existing literature, we remark that given the description of toric schemes over $S$ in terms of non-complete fans in $N_{\mathbb{R}}\times\mathbb{R}$, there is a strong analogy between toric schemes over $S$ and (non-complete) toric varieties over a field.  In this setting, the analogue of Theorem \ref{th:intro0} is well-known (see e.g. \cite[Theorem 1]{fmss}, \cite[Theorem~2.1]{Brion-equi} or \cite[Proposition 2.1]{FS}). However, up to the knowledge of the author, it has not been stated for toric schemes over $S$ before. 
 
 As an application of our results, in Section \ref{subsec:exa}, we compute the Chow groups of several toric schemes over $S$ and of their special fibers. 
 
As a final remark, note that in order to define the product of two cycles in a toric scheme, one can consider  Fulton's operational Chow \emph{cohomology} rings or its equivariant version. These are defined in terms of bivariant classes (see \cite[Chapter 17, 20]{fulint} and \cite[Section 6]{EG} for the equivariant version). In \cite{bot-ara} a combinatorial characterization of the equivariant operational Chow cohomology rings of a toric scheme and that of its special fiber are given in terms of piecewise polynomial functions.  Moreover, in \cite[Theorem 3.16]{bot-ara} it is shown that the equivariant Chow groups of regular, proper toric models satisfy Poincaré duality.

\subsection{Outline of the paper}
 In Section \ref{sec:toric-schemes} we recall the combinatorial characterization of toric schemes over $S$. 
In particular, we state the correspondence between complete SCR polyhedral complexes and isomorphism classes of proper toric schemes over $S$ (see Theorem~\ref{th:corr-comp}). We also remind the description of horizontal and vertical torus orbits in terms of the combinatorics of the polyhedral complex in Section \ref{subsec:orbits} and the characterization of torus invariant Cartier divisors in terms of piecewise affine functions in Section~\ref{subsec:t-divisors}. 
 
 In Section \ref{sec:chow-groups} we first give basic definitions and properties concerning the Chow (homology) groups of schemes over $S$. Then we focus on the toric case and we prove Theorems \ref{th:intro0} and \ref{th:intro1}. In Section \ref{subsec:chow-special} we give a combinatorial description of the Chow groups of the special fiber. In Section \ref{subsec:exa} we give several examples where we calculate the Chow groups of toric schemes and their special fibers explicitly.  Finally, in Section \ref{subsec:rank} we give a combinatorial formula for the rank of the Chow groups of a toric scheme over $S$.

 \subsection{Conventions}\label{sec:conventions} The following notation will be used during the whole article. The symbol $K$ will denote a field equipped with a non-trivial discrete valuation. We will denote by $R$ the corresponding valuation ring, by $\mathfrak{m}$ its maximal ideal with generator $\varpi$, and by $\kappa$ the residue field, which we will 
assume to be algebraically closed.

We set $S = \operatorname{Spec}\left(R\right)$ and denote by $\eta$ and by $s$ the generic and the special point of $S$, respectively. For a scheme $\mathcal{X}$ over $S$, $\mathcal{X}_{\eta} \coloneqq \mathcal{X} \otimes_S \operatorname{Spec}(K)$ and $\mathcal{X}_{s} \coloneqq \mathcal{X} \otimes_S \operatorname{Spec}(\kappa)$ denote the generic and the special fiber, respectively. 

$N$ denotes an $n$-dimensional lattice and $M = N^{\vee}$ its dual lattice. For any ring $F$ we write $N_F$ and $M_F$ for the tensor products $N \otimes_{\Z} F$,  respectively $M_{F} = M \otimes_{\Z} F$.

Finally, for the whole article, we will assume that our Chow groups are tensored with $\Q$. 

\vspace{0.5cm}
\noindent
\emph{Acknowledgements:} I am thankful to José Burgos and Roberto Gualdi for useful discussions and comments.  I am also very grateful to an anonymous referee for many corrections and fruitful comments of a previous version. 
 
\section{Toric schemes over $S$}\label{sec:toric-schemes}
The goal of this section is to recall the combinatorial characterization of toric schemes over $S$ in terms of polyhedral complexes.  We follow \cite[Chapter 3]{BPS} and \cite[Chapter IV]{KKMD}.
 \subsection{SCR polyhedral complexes}\label{sec:SCR}
 
 \begin{Def}
 A \emph{polyhedron} in $N_{\R}$ is a convex set defined as the intersection of a finite number of closed halfspaces. It is \emph{strongly convex} if it does not contain any line. A \emph{polyhedral cone} is a polyhedron $\sigma$ such that $\lambda \sigma = \sigma$ for all $\lambda \in \R_{>0}$. A \emph{polytope} is a bounded polyhedron. 
 \end{Def}
A polyhedron $\Lambda$ in $N_{\R}$ can be also described as the Minkowski sum of a polyhedral cone and a polytope, i.e. there exists a set of vectors $\left\{b_j\right\}_{j=1}^k$ and a non-empty set of points $\left\{b_j\right\}_{j = k+1}^{\ell}$ in $N_{\R}$ such that 
 \begin{align}\label{eq:v-rep}
 \Lambda = \cone\left(b_1, \dotsc, b_k\right) + \conv\left(b_{k+1}, \dotsc, b_{\ell}\right),
 \end{align}
 where 
 \[
 \cone\left(b_1, \dotsc, b_k\right) = \left\{\sum_{j=1}^k \lambda_jb_j \; \big{|} \; \lambda_j \in \R_{\geq 0} \right\}
 \]
 is the cone generated by the given set of vectors (we assume that $\cone(\emptyset) = \{0\}$) and 
 \[
 \conv\left(b_{k+1}, \dotsc, b_{\ell}\right) = \left\{\sum_{j=k+1}^{\ell} \lambda_j b_j \; \big{|} \; \lambda_j \in \R_{\geq 0}, \sum_{j = k+1}^{\ell}\lambda_j = 1 \right\}.
 \]
 is the convex hull of the given set of points. 
 \begin{Def}\label{def:pol-rat} A polyhedron $\Lambda$ is said to be rational if it admits a representation as in \eqref{eq:v-rep} with $b_j \in N_{\Q}$ for $j = 1, \dotsc, \ell$. 
 \end{Def}
 \begin{Def}
 Let $\Lambda$ be a polyhedron in $N_{\R}$ and consider a representation as in \eqref{eq:v-rep}. The \emph{recession of} $\Lambda$ is the polyhedral cone $\rec(\Lambda)$ in $N_{\R}$ defined by 
 \[
 \rec(\Lambda) \coloneqq \cone\left(b_1, \dotsc, b_k\right).
 \]
 \end{Def}
 Similar to polytopes, polyhedra have supporting hyperplanes, faces, facets, vertices, edges, etc. (see e.~g.~\cite[Section 18]{ROCK}). Note that in general, in contrast to a polytope, a polyhedron can have no vertices. However, in the case where the recession cone is strongly convex this cannot happen. Indeed, if $\Lambda$ is a polyhedron with strongly convex recession cone $\sigma$, then by \cite[Lemma~7.1.1]{CLS} we have that the set $V \coloneqq \left\{v \in \Lambda \; \big{|} \; v \text{ vertex }\right\}$ is finite and non-empty. Moreover, in this case, $\Lambda$ can be written as the Minkowski sum  $\Lambda = \conv(V) + \sigma$.

As in the polytope case, if $\Lambda$ is rational and full dimensional, then it has a unique facet presentation 
\begin{align}\label{eq:facet-rep}
\Lambda = \left\{m \in N_{\R} \; \big{|} \; \langle m, u_F \rangle \geq -a_F \; \forall \text{ facet } F \right\},
\end{align}
where $u_F \in M$ is the unique primitive inward pointing facet normal vector, and $a_F \in \Q$ (see \cite[Theorem 19.1]{ROCK}).

 \begin{Def}\label{def:cone-pol}
 Let $\Lambda$ a polyhedron in $N_{\R}$. The cone $c(\Lambda)$ is defined by 
 \[
 c(\Lambda) \coloneqq \overline{\R_{>0}\left(\Lambda \times \{1\}\right)} \subseteq N_{\R} \times \R_{\geq 0}.
 \]
 It is a closed cone in $N_{\R} \times \R_{\geq 0}$.
\end{Def}
We now gather some notions regarding polyhedral complexes. Here we follow mainly \cite[Chapter 2]{BPS}. 

\begin{Def}\label{def:SCR}
A polyhedral complex $\Pi$ in $N_{\R}$ is a collection of polyhedra in $N_{\R}$ satisfying the following conditions. 
\begin{enumerate}
\item Every face of an element of $\Pi$ is again an element in $\Pi$.
\item Every two elements of $\Pi$ are either disjoint or they intersect in a common face.
\end{enumerate}
The support of $\Pi$ is defined as the set 
\[
|\Pi| \coloneqq \bigcup_{\Lambda \in \Pi}\Lambda.
\]
For any integer~$k$, the \emph{$k$-skeleton} of a polyhedral complex $\Pi$ is the collection of all polyhedra in $\Pi$ of dimension~$k$.
We denote it by~$\Pi(k)$. 

The polyhedral complex $\Pi$ is said to be \emph{complete} if $|\Pi| = N_{\R}$. It is said to be \emph{strongly convex}, respectively \emph{rational}, if all of its elements are strongly convex, respectively rational. It is called \emph{conical} if all of its elements are polyhedral cones.  A strongly convex conical polyhedral complex is called a \emph{fan}.
\end{Def}
In the following, a SCR polyhedral complex  will stand for a strongly convex rational polyhedral complex. 
\begin{Def} Let $\Pi$ be a polyhedral complex in $N_{\R}$. The \emph{recession of} $\Pi$ is defined by
\[
\rec(\Pi) \coloneqq \left\{\rec(\Lambda) \; \big{|} \; \Lambda \in \Pi \right\}.
\]
The \emph{cone of} $\Pi$ is the conical polyhadral complex in $N_{\R} \times \R_{\geq 0}$ defined by 
\[
c(\Pi) \coloneqq \left\{c(\Lambda) \; \big{|}\; \Lambda \in \Pi \right\} \cup \left\{\sigma \times \{0\} \; \big{|}\; \sigma \in \rec(\Pi) \right \}.
\]
\end{Def}
\begin{rem} 

In general, it is not true that the recession or the cone of a given polyhedral complex is again a polyhedral complex (see \cite[Example 2.1.6]{BPS}). However, by the following proposition this is true in the complete case..

\end{rem}
\begin{prop}
Let $\Pi$ be a complete polyhedral complex in $N_{\R}$. Then $\rec(\Pi)$ and $c(\Pi)$ are complete conical polyhedral complexes in $N_{\R}$ and $N_{\R} \times \R_{\geq 0}$, respectively. If $\Pi$ is strongly convex, respectively rational, then $\rec(\Pi)$ and $c(\Pi)$ are fans, respectively rational. 
\end{prop}
\begin{proof}
This is \cite[Proposition 2.1.7 and Remark 2.1.12]{BPS}. 
\end{proof}
The following proposition relates polyhedral complexes in $N_{\R}$ with conical polyhedral complexes in $N_{\R} \times \R_{\geq 0}$ in the complete case. 
\begin{prop}\label{cor:cone1}
The correspondence 
\[
\Pi \mapsto c(\Pi)
\]
is a bijection between the set of complete polyhedral complexes in $N_{\R}$ and the set of complete conical polyhedral complexes in $N_{\R} \times \R_{\geq 0}$. Its inverse is the correspondence that to each complete conical polyhedral complex $\Gamma$ in $N_{\R} \times \R_{\geq 0}$ associates the polyhedral complex in $N_{\R}$ obtained by intersecting $\Gamma$ with the hyperplane $N_{\R} \times \{1\}$. These bijections preserve rationality and strong convexity.
\end{prop}
\begin{proof}
This is \cite[Corollary 2.1.13]{BPS}.
\end{proof}

\begin{Def}\label{def:smooth} A fan in $N_{\R}\times \R_{\geq 0}$ is \emph{smooth}, if each of its cones is smooth, i.e. generated by a subset of a basis of $N \times \Z$.  A complete SCR polyhedral complex $\Pi$ is \emph{smooth} if $c(\Pi)$ is smooth. 

\end{Def}
\begin{rem}
Our definition of a \emph{smooth} fan coincides with the one in \cite[Definition~1.2.16]{CLS}. In the literature, sometimes also the word \emph{regular} is used (for example in \cite{oda}). This is however not the same as the notion of a \emph{regular fan} in \cite[Definition~2.5.4]{BPS}.
\end{rem}

 \subsection{Toric schemes over a DVR}\label{subsec:toric-schemes}

We recall the combinatorial characterization of proper toric schemes over $S$ in terms of complete polyhedral complexes. 
We follow \cite[Section~3.5]{BPS} (see also \cite[Chapter IV.3]{KKMD}).

Assume that $N$ is the lattice of one-parameter subgroups of a torus $\mathbb{T}_K$. As before, $M = N^{\vee}$ denotes its dual lattice. Let $\widetilde{N} \coloneqq N \oplus \Z$ and $\widetilde{M} \coloneqq M \oplus \Z$ and set $\mathbb{T}_S \coloneqq \operatorname{Spec}R[M] \simeq \mathbb{G}_{m,S}^n$. This is a split torus over $S$. 

The following is \cite[Definition 3.5.1]{BPS}.  
\begin{Def}\label{def:toric-scheme}
A \emph{toric scheme over $S$} of relative dimension $n$ is a normal integral separated scheme of finite type $\mathcal{X}$ ovr $S$, equipped with a dense open embedding $\mathbb{T}_K \hookrightarrow \mathcal{X}_{\eta}$ and an $S$-action of $\mathbb{T}_S$ over $\mathcal{X}$ that extends the action of $\mathbb{T}_K$ on itself by translations. 
\end{Def}
Note that if $\mathcal{X}$ is a toric scheme over $S$, then $\mathcal{X}_{\eta}$ is a toric variety over $K$ with torus $\mathbb{T}_K$.  

The following is \cite[Definition 3.5.2]{BPS}.
\begin{Def}
Let $X$ be a toric variety over $K$ with torus $\mathbb{T}_K$ and let $\mathcal{X}$ be a toric scheme over $S$ with torus $\mathbb{T}_S$. We say that $\mathcal{X}$ is a \emph{toric model of $X$ over $S$} if the identity morphism on $\mathbb{T}_K$ can be extended to an isomorphism from $X$ to $\mathcal{X}_{\eta}$.
If $\mathcal{X}'$ and $\mathcal{X}$ are two toric models of $X$ over $S$ and $\alpha \colon \mathcal{X}' \to \mathcal{X}$ is an $S$-morphism, we say that $\alpha$ is a \emph{morphism of toric models} if its restriction to $\mathbb{T}_K$ is the identity. Two toric models are said to be \emph{isomorphic} if there exists a morphism of toric models between them which is an isomorphism.
\end{Def}
\begin{Def}
A toric model $\mathcal{X}$ is \emph{regular}, respectively \emph{proper},  if it is a regular, respectively proper, scheme. 
\end{Def} 
\begin{rem}\label{rem:base-change}
It follows from the Semi-stable reduction theorem II in \cite[Chapter IV Section 3]{KKMD} that any toric model $\mathcal{X}$ of $X$ over $S$ is dominated by a regular one (possibly after base change to $S' = \operatorname{Spec}(R')$ for a discrete valuation ring $R' \supseteq R$, integral over $R$).
\end{rem}

\subsection{Combinatorial characterization}\label{subsec:correspondence}

Let $\widetilde{\Sigma}$ be a rational fan in $N_{\R} \times \R_{\geq 0}$. To $\widetilde{\Sigma}$ we can associate a toric scheme over $S$, which we denote by $\mathcal{X}_{\widetilde{\Sigma}}$. This is done in the usual way by associating an affine toric scheme 
\[
\mathcal{X}_{\sigma} = \on{Spec}\left(R\left[\widetilde{M}_{\sigma}\right]/(\chi^{(0,1)}-\varpi)\right)
\]
to each $\sigma \in \widetilde{\Sigma}$ and then applying an appropriate gluing construction 
\begin{align}\label{eq:gluing}
\mathcal{X}_{\widetilde{\Sigma}} \coloneqq \bigcup_{\sigma \in \widetilde{\Sigma}}\mathcal{X}_{\sigma}.
\end{align}
Here we have set $\widetilde{M}_{\sigma} = \widetilde{M} \cap \sigma^{\vee}$ and $R\left[\widetilde{M}_{\sigma}\right]$ for the associated semigroup $R$-algebra. 

We gather here some properties regarding this construction. Details can be found in \cite[Section~3.5]{BPS}.

\begin{enumerate}
\item Given $\widetilde{\Sigma} \subseteq N_{\R} \times \R_{\geq 0}$, we set 
\[
\Pi \coloneqq \widetilde{\Sigma} \cap \left(N_{\R} \times \{1\}\right) \quad \text{and} \quad \Sigma \coloneqq \widetilde{\Sigma} \cap \left(N_{\R} \times \{0\}\right).
\]
This decomposes $\widetilde{\Sigma}$ in two different types of polyhedral cones. The ones in $\Sigma$ (i.~e.~those contained in the hyperplane $N_{\R} \times \{0\}$) and the ones that are cones over polyhedra in $\Pi$.  If $\widetilde{\Sigma}$ is complete in $N_{\R}\times \R_{\geq 0}$, then by Proposition \ref{cor:cone1}, the former cones are the ones coming from recession cones of the polyhedra in $\Pi$. In general, if $\sigma \in \Sigma$, then $\mathcal{X}_{\sigma}$ is contained in the generic fiber and it agrees with the affine toric varity $X_{\sigma}$. If $\sigma$ is not contained in $N_{\R} \times \{0\}$ then $\mathcal{X}_{\sigma}$ is not contained in the generic fibre.

\item Given polyhedra $\Lambda$, $\Lambda'$ in $\Pi$ with $\Lambda \subseteq \Lambda'$, we have a natural open immersion of affine toric schemes $\mathcal{X}_{\Lambda} \hookrightarrow \mathcal{X}_{\Lambda'}$. Moreover, if a cone $\sigma \in \Sigma$ is a face of a cone $c(\Lambda)$ for some $\Lambda \in \Pi$, then the affine toric variety $X_{\sigma}$ is also an open subscheme of $\mathcal{X}_{\Lambda}$. Then \eqref{eq:gluing} can be written as 
\[
\mathcal{X}_{\widetilde{\Sigma}} = \bigcup_{\Lambda \in \Pi}\mathcal{X}_{\Lambda} \cup \bigcup_{\sigma \in \Sigma}X_{\sigma}.
\]

\item There are open immersions
\[
\mathbb{T}_K \hookrightarrow \mathcal{X}_{\widetilde{\Sigma},\eta} \hookrightarrow \mathcal{X}_{\widetilde{\Sigma}}
\]
of schemes over $S$ and the action of $\mathbb{T}_S$ over $\mathcal{X}_{\widetilde{\Sigma}}$ extends that of $\T_K$ on itself. It follows that $\mathcal{X}_{\widetilde{\Sigma}}$ is a toric scheme over $S$. 
\item $\Sigma$ is a fan and defines a toric variety $X_{\Sigma}$ over $K$ which coincides with the generic fiber $\mathcal{X}_{\widetilde{\Sigma},\eta}$. It follows that $\mathcal{X}_{\widetilde{\Sigma}}$ is a toric model of $X_{\Sigma}$.
\item The special fiber $\mathcal{X}_{\widetilde{\Sigma},s}$ has an induced action by $\mathbb{T}_{\kappa}$ but, in general, it is not a toric variety over $\kappa$ (it is not necessarily irreducible nor reduced). It is reduced if and only if the vertices of all $\Lambda \in \Pi$ are contained in the lattice $N$, and in this case the reduced schemes associated to its irreducible components are toric varieties over $\kappa$ with this action (see Subsection \ref{subsec:orbits}). 

\item If the fan $\widetilde{\Sigma}$ is complete, then the scheme $\mathcal{X}_{\widetilde{\Sigma}}$ is proper over $S$. In this case, the set $\left\{\mathcal{X}_{\Lambda}\right\}_{\Lambda \in \Pi}$ is an open cover of $\mathcal{X}_{\widetilde{\Sigma}}$.

\end{enumerate}

In fact, one can show that toric schemes over $S$ of relative dimension $n$ are classified by rational fans in $N_{\R} \times \R_{\geq 0}$. The following follows from \cite[Chapter~IV Section 3]{KKMD} and Proposition \ref{cor:cone1} (see also \cite[Theorem~3.5.3]{BPS}. The statement about the regularity is \cite[item (d) at pg.~192]{KKMD}). 
%
\begin{theorem}\label{th:corr-comp}
The correspondence 
\[
\Pi \longmapsto \mathcal{X}_{c(\Pi)}
\]
gives a bijection between complete SCR polyhedral complexes in $N_{\R}$ and isomorphism classes of proper toric schemes over $S$ of relative dimension $n$. Moreover, $\Pi$ is smooth if and only if $\mathcal{X}_{c(\Pi)}$ is regular.

\end{theorem}

Fix a complete fan $\Sigma$ in $N_{\R}$. 
\begin{cor}\label{th:models}
The correspondence 
\[
\Pi \longmapsto \mathcal{X}_{c(\Pi)}
\]
gives a bijection between complete SCR polyhedral complexes in $N_{\R}$ such that $\rec(\Pi) = \Sigma$ and equivariant isomorphism classes of proper toric models of $X_{\Sigma}$ over $S$. As before $\Pi$ is smooth if and only if $\mathcal{X}_{c(\Pi)}$ is regular.

\end{cor}

Given a complete polyhedral complex $\Pi$, in order to simplify notation, we write $\mathcal{X}_{\Pi}$ for $\mathcal{X}_{c(\Pi)}$. 

\begin{Def}
Note that $\Sigma$ itself is a complete SCR polyhedral complex in $N_{\R}$ with $\rec(\Sigma) = \Sigma$. The toric scheme $\mathcal{X}_{\Sigma}$ is hence a model over $S$ of $X_{\Sigma}$ which is called the \emph{canonical model}.
\begin{rem} Note that the special fiber $\mathcal{X}_{\Sigma,s}$ of the canonical model is the toric variety over $\kappa$ defined by the fan $\Sigma$.  Hence, in this case, the special fiber is in particular both irreducible and reduced.
\end{rem}

\end{Def}
\begin{rem}
Consider the set $R(\Sigma)$ of all complete smooth SCR polyhedral complexes $\Pi$ such that $\operatorname{rec}(\Pi) = \Sigma$.  It is endowed with the partial order given by
\[
\Pi' \geq \Pi \quad \text{iff} \quad \text{ $\Pi'$ is a subdivision of $\Pi$}.
\]
Given two elements in $R(\Sigma)$ there is always a third one dominating both. This follows from the Semi-stable reduction theorem II in \cite[Chapter IV Section 3]{KKMD}. Hence this set is directed.
We note that the canonical model cooresponds to a minimal element in $R(\Sigma)$. 

On the other hand, given $\Pi' \geq \Pi$ in $R(\Sigma)$ we have an induced proper morphism of toric models $\pi \colon \X_{\Pi'} \to \X_{\Pi}$ and conversely, any such morphism is induced by a subdivision(see \cite[item (f) at pg.~193]{KKMD}).  Hence there is a one-to-one correspondence between $R(\Sigma)$ and the directed set of regular proper toric models of $X_{\Sigma}$. The set $R(\Sigma)$ is used in \cite{bot-ara}  to study inverse and direct limits of Chow groups of toric models. 
\end{rem}

\subsection{Torus orbits}\label{subsec:orbits}

Let $\Pi$ be a complete smooth SCR polyhedral complex in $N_{\R} \times\{1\}$ and let $\X_{\Pi}$ be the corresponding proper regular toric scheme over $S$. Set $\Sigma = \rec(\Pi)$. We view $\Sigma$ in $N_{\R} \times \{0\}$. We recall the combinatorial characterization of torus orbits in $\X_{\Pi}$ given in \cite[item (e') at pg.~192]{KKMD} and in \cite[Section~3.5]{BPS}.

There are two kinds of torus orbits. 
\begin{itemize}
\item First, there is a bijection between cones in $\Sigma$ and the set of orbit closures under the action of $\T_K$ on $\X_{\Pi, \eta}$ taking $\sigma$ to the Zariski closure $V(\sigma)$ in $\X_{\Pi}$ of the orbit $O(\sigma)\subseteq \X_{\Pi,\eta} = X_{\Sigma}$. Then $V(\sigma)$ is a horizontal scheme over $S$ in the sense that the structure morphism $V(\sigma) \to S$ is dominant, of relative dimension $n-\dim(\sigma)$. 

Moreover, $V(\sigma)$ has the structure of a proper toric scheme over $S$. Indeed, let $N(\sigma) = N/(N \cap \R\sigma)$ and $\pi_{\sigma} \colon N_{\R} \to N(\sigma)_{\R}$ the linear projection. Consider the following collection of subsets in $N(\sigma)_{\R}$
\[
\Pi(\sigma) = \left\{ \pi_{\sigma}(\Lambda) \; \big{|} \; \Lambda \in \Pi \; , \; \sigma \subseteq \rec(\Lambda) \right\}.
\]
Then $\Pi(\sigma)$ is a complete SCR polyhedral complex in $N(\sigma)_{\R}$ and we have an isomorphism of toric schemes over $S$
\[
V(\sigma) \simeq \X_{\Pi(\sigma)}
\]
(see \cite[Proposition 3.5.7]{BPS}).
\item On the other hand there is a bijection between $\Pi$ and the set of orbits under the action of $\T_{\kappa}$ on the special fiber $\X_{\Pi,s}$. Let $\widetilde{N} = N \oplus \Z$. Given $\Lambda \in \Pi$, set $\widetilde{N}(\Lambda) = \widetilde{N}/(\widetilde{N} \cap \R c(\Lambda))$, $\widetilde{M}(\Lambda) = \widetilde{N}(\Lambda)^{\vee}$ and $\pi_{\Lambda} \colon \widetilde{N}_{\R} \to \widetilde{N}(\Lambda)_{\R}$ the linear projection. Then to $\Lambda$ one associates a vertical subscheme $V(\Lambda)$ (vertical in the sense that it is contained in the special fiber). We refer to \cite[pg.~98]{BPS} for the precise definition. It has the structure of a toric variety over $\kappa$ with torus $\T(\Lambda) = \operatorname{Spec}\left(\kappa\left[\widetilde{M}(\Lambda)\right]\right)$ and corresponding fan given by 
\[
\Pi(\Lambda) = \left\{\pi_{\Lambda}(c(\Lambda')) \; \big{|} \; \Lambda' \in \Pi \; , \; \Lambda \prec \Lambda' \right\}
\]
in $\widetilde{N}(\Lambda)_{\R}$,
where \enquote{$\prec$}  stands for \enquote{is a face of} (see \cite[Proposition 3.5.8]{BPS}). In particular, it has absolute dimension $n-\dim\Lambda$. 

There is a one-to-one correspondence between the vertices of $\Pi$ and the irreducible components of the special fiber given by 
\begin{align}\label{rem:vertex}
v \mapsto V(v).
\end{align}
The torus orbits contained in $V(v)$ correspond to the polyhedra $\Lambda \in \Pi$ containing $v$. In particular, the components given by two vertices $v$ and $v'$ share an orbit of absolute dimension $\ell$ if and only if there exists a polyhedron of dimension
$n - \ell$ containing both $v$ and $v'$.

It follows from  \cite[Example 3.6.11]{BPS} that the special fiber is reduced if and only if all the vertices of $\Pi$ are in $N$.  
\end{itemize}
We will frequently use the following definitions.
\begin{Def} Let $\Pi$ be a complete smooth SCR polyhedral complex in $N_{\R} \times\{1\}$ and $\Sigma = \rec(\Pi)$. We view $\Sigma$ in $N_{\R} \times \{0\}$. Let $0 \leq k \leq n-1$. Fix $\tau \in \Sigma(k)$ and $\Lambda \in \Pi(k)$.  
\begin{enumerate}
\item For any $\sigma \in \Sigma(k+1)$ with $\tau \prec \sigma$, we let $v_{\sigma/\tau}$ be the corresponding \emph{lattice normal vector}. This is the primitive generator of the image of $\sigma$ under the map $\pi_{\tau}\colon N_{\R} \to N(\tau)_{\R}$.  If $k = 0$ we simply write $v_{\sigma}$ for $v_{\sigma/\{0\}}$. 

\item For $\Gamma \in \Pi(k)$ with recession cone $\tau$, we let $v_{\Gamma}$ be the image of $\Gamma$ under $\pi_{\tau}$. It is a vertex of the polyhedral complex $\Pi(\tau)$ (see Remark \ref{rem:dim} below).  

\item For any $\Lambda' \in \Pi(k+1)$ with $\Lambda \prec \Lambda'$, we denote by $v_{\Lambda'/\Lambda}$ the primitive generator of the ray generated by the image of $c(\Lambda')$ under the map $\pi_{\Lambda}\colon \widetilde{N}_{\R} \to \widetilde{N}(\Lambda)_{\R}$. 
\end{enumerate}

\end{Def}
\begin{rem}\label{rem:dim} In item (2) above,  the assumption that $\tau$ and $\Gamma$ have the same dimension is necessary, since otherwise the image of $\Gamma$ under $\pi_{\tau}$ might not be a point.
\end{rem}

\subsection{$\T_S$-invariant divisors on toric schemes}\label{subsec:t-divisors}

As before, let $\X = \X_{\Pi}$ be the complete regular toric scheme over $S$ associated to a complete smooth SCR polyhedral complex $\Pi$ in $N_{\R}$ and let $\Sigma = \rec(\Pi)$. Assume that the vertices of $\Pi$ are in the lattice $N$. 

We will consider $\Q$-Weil divisors on $\X$. There are two morphisms from $\T_S \times \X$ to $\X$: the torus action, denoted by $\mu$ and the second projection, denoted by $p_2$. A $\Q$-Weil divisor $D$ on $\X$ is called $\T_S$-invariant if $\mu^*D = p_2^*D$.

As in the case of toric varieties over a field, if $\X$ is regular, then $\T_S$-invariant divisors on $\X$ can be characterized by some piecewise affine functions. We explain this briefly.

A \emph{rational piecewise affine function on} $\Pi$ is a piecewise affine function $\phi$ on $\Pi$ such that on each $\Lambda \in \Pi$ we have 
\[
\phi|_{\Lambda}(u) = \langle m_{\Lambda},u\rangle + \ell_{\Lambda}
\]
with $\left(m_{\Lambda},\ell_{\Lambda}\right) \in M_{\Q} \times \Q$ (see \cite[section 2.5 and Proposition 2.6.7]{BPS}). 

In this case, its recession function $\psi = \rec(\phi)$ is the rational piecewise linear function on $\Sigma$ given by 
\[
\psi|_{\on{rec}(\Lambda)}(u) = \langle m_{\Lambda},u\rangle.
\]
Now, to $\phi$ we want to associate a $\T_S$-invariant $\Q$-Weil divisor on $\X_{\Pi}$.  For this, we have to take into account that we have two types of orbits. Each vertex $v \in \Pi(0)$ defines a vertical invariant divisor $V(v)$ and every ray $\tau \in \Sigma(1)$ defines a horizontal prime Weil divisor $V(\tau)$. 
\begin{Def} 
Let $\phi$ be a rational piecewise affine function on $\Pi$. Then one sets 
\[
D_{\phi} = \sum_{v \in \Pi(0)}-\phi(v)V(v) + \sum_{\tau \in \Sigma(1)}-\psi(v_{\tau})V(\tau).
\]
\end{Def}
This is a $\T_S$-invariant $\Q$-Weil divisor on $\X_{\Pi}$.


The following is a reformulation of \cite[Theorem 3.6.7 and Proposition 3.6.10]{BPS} 
\begin{theorem}\label{th:div-lin}
Let $\Pi$ and $\X_{\Pi}$ as above. The correspondence $\phi \mapsto D_{\phi}$ is an isomorphism between the group of rational piecewise affine functions on $\Pi$ and the group of $\T_S$-invariant $\Q$-Weil divisors on $\X_{\Pi}$. Moreover, if $\phi_1$ and $\phi_2$ are two rational piecewise affine functions on $\Pi$ then $D_{\phi_1}$ is rationally equivalent to $D_{\phi_2}$ if and only if $\phi_1- \phi_2$ is affine.
\end{theorem}

\section{Chow groups of toric schemes}\label{sec:chow-groups}

Recall our notations from Section \ref{sec:conventions}. In particular, recall that we assume that our Chow groups are tensored with $\Q$.

\subsection{Chow (homology) groups}
In this subsection we recall some terminology regarding Chow groups of schemes over $S$. Our main references is~\cite[Chapter~20]{fulint}.
Let $\mathcal{X}$ be a separated scheme of finite type over $S$.

For a closed integral subscheme $\mathcal{V}$ of $\mathcal{X}$ over~$S$, its \emph{$S$-(absolute) dimension} is defined as
\[
\dim_S\mathcal{V}:=\text{tr.}\deg (K(\mathcal{V})/K(T))-\codim(T,S) +1,
\]
where $T$ is the closure of the image of $\mathcal{V}$ in $S$ by the structural morphism (see Remark \ref{rem:dimension} below for the reason of the name $S$-(absolute) dimension). 

For example, it follows from \cite[Lemma 20.1]{fulint} that if $\ca W \subseteq \ca V$ is a closed integral subscheme of $\ca V$, then 
\[
\dim_S\ca V = \dim_S\ca W + \codim(\ca W, \ca V),
\]
and if $\ca V^{\circ}\subseteq \ca V$ is a non-empty open subscheme, then 
\[
\dim_S\ca V^{\circ}= \dim_S\ca V.
\]

We remark that this notion coincides with the notion of (relative) dimension given in \cite[Chapter 17 and 20]{fulint} plus 1.  We choose to work with this absolute notion because it coincides with the one discussed in~\cite[section~(1.2)]{BGS} and in \cite{bot-ara}.  (See also \cite[Example 29.1.5]{fulint} where this notion is suggested). 
\begin{Def}\label{def:chow}
For any integer $k$, we write $\CH_k(\X/S)$ for the Chow (homology) group of $S$-absolute dimension $k$ algebraic cycles modulo rational equivalence.  We will consider rational coefficients, i.e.~we always consider the Chow groups tensored with $\Q$.  We let
\[
\CH_*(\X/S) \coloneqq \bigoplus_k\CH_k(\X/S).
\]
\end{Def}

\begin{rem}\label{rem:dimension}
Note that we have 
\begin{eqnarray*}
\CH_k(\X_s/S) &=& \CH_{k} (\X_s/\kappa) \quad \text{ and }\\
\CH_k(\X_{\eta}/S) &=& \CH_{k-1} (\X_{\eta}/K),
\end{eqnarray*}
where $\CH_{k} (\X_s/\kappa)$ and $\CH_{k-1} (\X_{\eta}/K)$ denote the classical Chow groups for schemes defined over a field (i.e.  $k$-(resp. ($k-1$))-dimensional cycles modulo rational equivalence). This is the reason why we call this the $S$-absolute dimension instead of simply the absolute dimension.
\end{rem}

By \cite[Section 20.1]{fulint} one has covariance of the Chow groups $\CH_k$ for proper morphisms and contravariance for flat morphisms (of some relative dimension).

\subsubsection*{A localization sequence}

Consider the immersion  of the special fiber $\iota \colon \X_s \hookrightarrow \X$. This is a regular closed immersion of codimension $1$. Further, let $j \colon \X_{\eta} \hookrightarrow \X$ the open immersion of the generic fiber.  Then it follows from \cite[Proposition~1.8]{fulint} and the observation \cite[end of page~394]{fulint} that we have a \emph{localization sequence} 
 
\begin{eqnarray}\label{eq:localization}
\begin{tikzcd}
\CH_k(\X_s/S)\arrow["{\iota}_*"]{r} &\CH_k(\X/S)\arrow["{j}^*"]{r} &\CH_{k}(\X_{\eta}/S)\arrow{r} &0.
\end{tikzcd}
\end{eqnarray}

\begin{prop}\label{lem:zero}
Assume that $\X$ is irreducible and that the special fiber $\X_s$ has Chow group of zero cycles of rank one, i.e. that the degree map $\on{deg}\colon \CH_0(\X_s/\kappa) \to \Q$ is an isomorphism.
Then~$\CH_{0}(\X/S) = \{0\}$.
\end{prop}
\begin{proof}
One has to show that any closed integral subscheme $p$ of $\X$ of $S$-absolute dimension $0$ vanishes in $\CH_{0}(\X/S)$.

The localization sequence for $k = 0$ is of the form 
\[
\begin{tikzcd}
\CH_0(\X_s/\kappa)\arrow["{\iota}_*"]{r} &\CH_0(\X/S)\arrow["{j}^*"]{r} &0.
\end{tikzcd}
\]
Hence, $[p]= \iota_*(\gamma)$ for some $\gamma \in \CH_0(\X_s/\kappa)$.
By the assumption on universal triviality, we have that there is a rational number $\ell \in \Q$ such that
\[
\gamma = \ell[s] \; \text{ in }\; \CH_0(\X_s/\kappa),
\]
where $s$ denotes the closed point of $S$. Let $m \in \Z$ such that $m \cdot \ell$ is an integer. We have
\[
\frac{1}{m}\on{div}(\varpi^{m \cdot \ell }) = \frac{1}{m}\on{lenght}_R\left(R/\varpi^{m \cdot \ell}\right)\cdot [s] =  \ell[s].
\]
Hence $[p] = \iota_*(\gamma) = 0$ in $\CH_0(\X_s/\kappa)$. 
\end{proof} 

We will see later in Proposition \ref{prop:ch-zero} that if $\X$ is a proper regular \emph{toric} scheme over $S$ then the Chow group of zero cycles of the special fiber $\X_s$ has rank $1$ (see Corollary \ref{cor:special-toric}).

\subsubsection*{The specialization map}
Let $\iota \colon \X \hookrightarrow \mathfrak{Y}$ be a regular codimension $d$-embedding of $S$-schemes (this means that $\iota$ is a closed immersion, defined locally by a regular sequence of length $d$) and consider a commutative diagram of $S$-schemes 
\begin{figure}[H]
\begin{center}
    \begin{tikzpicture}
      \matrix[dmatrix] (m)
      {
        \X' & \mathfrak{Y}'\\
       \X & \mathfrak{Y} \\
      };
      \draw[->] (m-1-1) to  (m-1-2);
      \draw[ ->] (m-1-1)--(m-2-1);
      \draw[->] (m-1-2)-- (m-2-2);
      \draw[->] (m-2-1) to node[above]{$\iota$} (m-2-2);
      
     \end{tikzpicture}
     \end{center}
     \end{figure} 
Then one has a \emph{refined Gysin map}
\[
\iota^{!}\colon \CH_*(\mathfrak{Y}'/S) \longrightarrow \CH_*\left(\X'/S\right)
\]
with 
\[
\iota^{!}\left(\CH_k(\mathfrak{Y}'/S)\right) \subseteq \CH_{k-d}\left(\X'/S\right).
\]
For details and the construction of the refined Gysin map we refer to \cite[Chapter 6]{fulint}. 
As special case is the following: if $\ca{V}$ is a closed subscheme of $\mathfrak{Y}$ such that $\ca{V} \times_{\mathfrak{Y}}\X \hookrightarrow \ca{V}$ is again a regular embedding of codimension $d$, then we have that
\begin{equation}\label{eq:gysin}
\iota^{!}\left([\ca{V}]\right) = \left[\ca{V} \times_{\mathfrak{Y}} \X \right].
\end{equation}

If we set $\X = \X' = \X_s$, $\mathfrak{Y} = \mathfrak{Y}' = \X$ and $\iota \colon \X_s \hookrightarrow \X$, then we have a refined Gysin morphism 
\[
\iota^! \colon \CH_k(\X/S) \to \CH_{k-1}(\X_s/S)
\]
 and the composition 
\[
\iota^{!} \circ {\iota}_* \colon \CH_k\left(\X_s/S\right) \to \CH_{k-1}\left(\X_s/S\right)
\]
 is the zero map \cite[Theorem 6.3]{fulint}. Hence by \eqref{eq:localization}, one obtains a uniquely defined \emph{specialization map} 
\[
\sp \colon CH_{k}\left(\X_{\eta}/S\right) \longrightarrow \CH_{k-1}\left(\X_s/S\right)
\]
such that for every irreducible subset $\ca V$ of $\X$ we have 
\[
\sp\left(\left[\ca V \cap \X_{\eta}\right]\right) = \iota^{!}\left([\ca V]\right).
\]
\begin{rem}
The map specialization map can be reinterpreted as a map 
\[
\sp\colon CH_{k-1}\left(\X_{\eta}/K\right) \longrightarrow \CH_{k-1}\left(\X_s/\kappa\right),
\]
whence the name. 
\end{rem} 

One can show that the specialization map is compatible with proper push-forward and flat pull-backs (see \cite[Proposition 20.3]{fulint} for the precise statements).

For example, if $V$ is an irreducible closed subset of $X_{\eta}$ and $\overline{V}$ is the closure of $V$ in $\X$ then the closed fiber $\overline{V}_{\kappa} = \overline{V} \times_{\X}\X_s \hookrightarrow \overline{V}$ is a regular embedding of codimension one, and $\sp\left([V]\right) = \left[\overline{V}_{\kappa}\right]$. In this case, the specialization map is induced by a map on the level of cycles.

\subsection{Chow groups of toric schemes}\label{sec:chow-hom}

We now consider the Chow (homology) groups of toric schemes over $S$. For this, we fix a complete smooth SCR polyhedral complex $\Pi$ in $N_{\R}$ and let $\X = \X_{\Pi}$ the corresponding proper, regular toric scheme over $S$ of $S$-absolute dimension $n+1$.

In this section we state and prove our main results. 

\begin{theorem}\label{th:inv-cycles}

Let $k$ be an integer. Then $\CH_k(\X/S)$ is generated by torus invariant horizontal and vertical cycles of $S$-absolute dimension $k$. A collection of generators is given by the family of horizontal cycles $\left[V(\sigma)\right]$ for $\sigma \in \Sigma(n-k+1)$ and of vertical cycles $\left[V(\Lambda)\right]$ for~$\Lambda \in \Pi(n-k)$. 
\end{theorem}
\begin{proof}
Let $\iota \colon \X_s \to \X$ denote the closed regular immersion of codimension one of the special fiber, and let $j \colon \X_{\eta}\to \X$ the open immersion of the generic fiber.
Consider the localization sequence \eqref{eq:localization}
\begin{eqnarray*}
\begin{tikzcd}
\CH_k(\X_s/S)\arrow["{\iota}_*"]{r} &\CH_k(\X/S)\arrow["{j}^*"]{r} &\CH_{k}(\X_{\eta}/S)\arrow{r} &0
\end{tikzcd}
\end{eqnarray*}
Equivalently,  and with the identifications from Remark \ref{rem:dimension}, we may consider the exact sequence 
\begin{eqnarray}\label{eq:loc-toric}
\begin{tikzcd}
\CH_{k}(\X_s/\kappa)\arrow["{\iota}_*"]{r} &\CH_{k}(\X/S)\arrow["{j}^*"]{r} &\CH_{k-1}(\X_{\eta}/K)\arrow{r} &0.
\end{tikzcd}
\end{eqnarray}
Note that if the two extremes are finitely generated groups, so is the middle one, and a family of generators for $\CH_k(\X/S)$ is given by the images of the generators of $\CH_{k}(\mathcal{X}_s/\kappa)$ and by a choice of preimages of the generators of~$\CH_{k}(\mathcal{X}_\eta/K)$.

Now, horizontal cycles of $S$-absolute dimension $k$ are by definition closures in $\X$ of $(k-1)$-dimensional cycles in the generic fiber $\X_{\eta} \simeq X_{\Sigma}$ (seen as a toric variety over the field $K$), and as is well known in toric geometry \cite[Proposition 2.1 (a)]{FS} that these cycles  are generated by the torus invariant classes $V(\sigma)$ for $\sigma \in \Sigma(n-k+1)$.

On the other hand, it follows from the description of the torus orbits given in Section \ref{subsec:toric-schemes} that the special fiber $\X_{s}$ is a finite union of toric varieties over $\kappa$
\[
\X_{s} = \bigcup_{\Lambda \in \Pi} X_{\Pi(\Lambda)}
\]
 glued in a way that is compatible with the polyhedral structure in $\Pi$. We claim that for any $\ell \geq 0$ the group $\CH_{\ell}\left(\X_{s}/\kappa\right)$ is generated by torus orbit closures of $\kappa$-dimension $\ell$. This would prove the proposition. In order to see this, let $\X_{s,i}$ consist of orbits of dimension at most $i$, i.e.
 \[
 \X_{s,i} = \bigcup_{\overset{\Lambda \in \Pi}{\codim(\Lambda) \leq i}}X_{\Pi(\Lambda)}.
 \]
 Then $\X_{s,i}$ is a closed subset of $\X_s$ and $\X_s$ admits the filtration 
 \[
 \X_s = \X_{s,n} \supseteq \X_{s,n-1} \supseteq \dotsc, \supseteq \X_{s,-1} = \emptyset.
 \]
 Then, by \cite[Proposition 1.8]{fulint}, for every $i$, we have an exact sequence 
 \begin{eqnarray}\label{eq:induction}
 \CH_{\ell}\left(\X_{s,i-1}/\kappa\right) \longrightarrow \CH_{\ell}\left(\X_{s,i}/\kappa\right) \longrightarrow \CH_{\ell}\left(\X_{s,i} \setminus \X_{s, i-1} /\kappa\right) \longrightarrow 0.
 \end{eqnarray}
 The claim then follows by induction on $i$. Indeed, the claim for the generators of the term in the left hand side of \eqref{eq:induction} follows by induction, and for the term on the right hand side, note that this is a disjoint union of toric varieties, for which the claim on the generators follows again from \cite[Proposition 2.1 (a)]{FS}.

\end{proof}

\begin{exa}\label{ex:ch-max}
 Consider the case $k = n+1$. Then $\CH_{n+1}(\X_s/ \kappa) = 0$ and by Equation \eqref{eq:loc-toric} we have 
 \[
\CH_{n+1}(\X/S) \simeq \CH_n(\X_{\eta}/K) \simeq \Q,
\]
generated by the closure of the dense torus $\T_K$ which corresponds to the zero cone $\{0\} \in \Sigma$. 
\end{exa}
\begin{exa}
The specialization map
\[
\on{sp} \colon \CH_k(\X_{\eta}/K) \longrightarrow  \CH_{k}(\X_s/\kappa).
\]
has the following combinatorial description in the toric case. Note that both the Gysin morphism $\iota^!$ and the proper pushforward $\iota_*$ preserve the stratification by torus orbits, hence also the specialization map. In other words, $\sp$ is \emph{toroidal}, induced by a map of polyhedral complexes, i.\,e.\,it maps a cone $\sigma$ to a polyhedron $\Lambda$ with $\rec(\Lambda) = \sigma$. 

`We obtain that $\sp$ is given on the level of cycles by 
\[
[V(\sigma)] \longmapsto \sum_{\Lambda} \on{mult}(\Lambda)[V(\Lambda)],
\]
where $\sigma \in \Sigma(n-k)$ and the sum is over all $\Lambda \in \Pi(n-k)$ with $\on{rec}(\Lambda) = \sigma$.  Here, the term $\on{mult}(\Lambda)$ is the \emph{multiplicity} associated to a polyhedron $\Lambda \in \Pi$ \cite[Def. 3.5.11 and Lemma 3.5.12]{BPS}, defined by 
\[
\on{mult}(\Lambda) \coloneqq \min\left\{n \geq 1 \; | \; \exists p \in \on{aff}(\Lambda) \text{ s.~t. }np \in N\right\},
\]
where $\on{aff}(\Lambda)$ denotes the affine span of $\Lambda$.  This multiplicity agrees with the intersection multiplicity of $V(\Lambda)$ in $\CH_k(\X_s/\kappa)$ (see \cite[Proposition 3.7.8]{BPS}).
\end{exa}
By construction of the toric scheme $\X$, rational functions on $\X$ are generated by elements $u\varpi^{\ell} \chi^m$, where $(m, \ell) \in M \times \Z$ and $u$ is a unit in $R$. The following lemma computes the divisor of such a generating element in terms of invariant divisors. We write $\omega$ for the primitive vector spanning a ray $\tau_{\omega} \in c(\Pi)(1)$ and denote by $D_{\omega}$ the corresponding invariant divisor. 
\begin{lemma}\label{lem:rel}
Let notations be as above and let $f $ be the rational function $u\varpi^{\ell} \chi^m$. Then 
\[
\on{div}(f) =\sum_{\omega}b_{\omega}D_{\omega},
\]
where 
\[
b_{\omega} = \begin{cases} \langle m , v\rangle_N \; & \text{ if }\omega= (v, 0) \in N_{\R}\times\{0\}, \\
\langle m , v\rangle_N + \ell \; & \text{ if } \omega = (v,1) \text{ for a vertex $v \in \Pi(0)$}. \end{cases}
\]
Here, $\langle \cdot, \cdot \rangle_N$ denotes the pairing between $M_{\R}$ and $N_{\R}$.  In other words, we have 
\[
\on{div}(f) =\sum_{\tau \in \Sigma(1)} \left\langle m,v_{\tau}\right\rangle_N V(\tau) + \sum_{v \in \Pi(0)}\left(\langle m,v\rangle + \ell\right) V(v) \in \Div(\X).
\]
\end{lemma}
\begin{proof}
The vanishing order of $f$ on $D_{\omega}$ is $\left(\left\langle\left(m, \ell\right),\omega \right\rangle\right)$, where the pairing is the one between $M_{\R} \times \R$ and $N_{\R} \times \R$. Then the lemma follows from distinguishing the two cases.
\end{proof}
By the formula in Theorem \ref{th:div-lin} we thus obtain the following corollary.

We are now ready to give an explicit combinatorial presentation of the Chow groups $\CH_k(\X/S)$ in terms of generators and relations for a given integer $k$. For $\tau \in \Sigma$ we write $M(\tau)$ for the dual $N(\tau)^{\vee}$. Consider the sequence 

\begin{align}\label{eq:chow}
&\bigoplus_{\tau \in \Sigma(n-k)}\left(M(\tau)_{\Q} \times \Q\right)\bigoplus \bigoplus_{\Lambda \in \Pi(n-k-1)}\widetilde{M}(\Lambda)_{\Q} \xrightarrow{\alpha} \\
& \xrightarrow{\alpha} \Q^{|\Sigma(n-k+1)|} \bigoplus \Q^{|\Pi(n-k)|} \xrightarrow{\beta} \CH_{k}(\X/S) \to 0 \nonumber,
\end{align}

with maps  $\alpha$ and $\beta$ given as follows. 
Let $\tau \in \Sigma(n-k)$. Recall that $V(\tau)$ corresponds to a toric $S$-scheme of $S$-absolute dimension $k+1$. Then $(m,\ell) \in M(\tau)_{\Q} \times \Q$ is mapped by $\alpha$ to 
\[
\sum_{\substack{\sigma \in \Sigma(n-k+1)\\ \tau \prec \sigma}}\left\langle m, v_{\sigma/\tau}\right\rangle V(\sigma) + \sum_{\substack{\Lambda \in \Pi(n-k)\\ \rec(\Lambda)= \tau}} \left(\left\langle m, v_{\Lambda}\right\rangle +\ell \right) V(\Lambda).
\]
 On the other hand, if $\Lambda \in \Pi(n-k-1)$, then $\Lambda$ corresponds to a vertical cycle $V(\Lambda)$ of $S$-absolute dimension $k+1$. This is isomorphic to a toric variety $X_{\Pi(\Lambda),\kappa}$ of $\kappa$-dimension $k+1$ whose torus has lattice of one-parameter subgroups $\widetilde{M}(\Lambda)$. Then $\alpha$ sends $m \in \widetilde{M}(\Lambda)_{\Q}$ to 
 \[
 \sum_{\substack{\Lambda' \in \Pi(n-k)\\ \Lambda \prec \Lambda'}} \left\langle m, v_{\Lambda'/ \Lambda} \right\rangle V(\Lambda').
 \]
 The map $\beta$ is just the canonical projection mapping a cycle corresponding to $\sigma \in \Sigma(n-k+1)$ or to $\Lambda' \in \Pi(n-k)$ to its class.
 \begin{theorem}\label{th:exact}
The sequence \eqref{eq:chow} is exact.
\end{theorem}
\begin{proof}
Surjectivity follows from Theorem \ref{th:inv-cycles}. On the other hand, by Lemma~\ref{lem:rel}, the maps are given by choosing an invariant subvariety $\mathfrak{Z}$ of $\X$ of $S$-absolute dimension $k+1$, choosing a rational function on $\mathfrak{Z}$ and taking its divisor. Hence, by definition, these will give relations in $\CH_{k}(\X/S)$ and hence the composition $\beta \circ \alpha$ is zero. It remains to show that the kernel of $\beta$ is contained in the image of $\alpha$. 

For this, we will show that all of the relations come from divisors on $\T_S$-invariant subvarieties. Let $Z_k^{\T_S}(\X)$ denote the group of invariant cycles on $\X$ and let $R_k^{\T_S}(\X)$ the subgroup generated by divisors of invariant rational functions on $\T_S$-invariant $(k+1)$-subvarieties of $\X$. Set $\CH^{\T_S}_k(\X/S) = Z_k^{\T_S}(\X)/R_k^{\T_S}(\X)$. We claim that the natural map 
\begin{eqnarray}\label{eq:inv}
\CH^{\T_S}_k(\X/S) \to \CH_k(\X/S)
\end{eqnarray}
is an isomorphism. Note that this would imply that all relations come from divisors on invariant subvarieties and hence this would imply the theorem. 

To prove the claim, consider the following commutative diagram with exact rows
\begin{figure}[H]
\begin{center}
    \begin{tikzpicture}
      \matrix[dmatrix] (m)
      {
        \CH_k^{\T_S}(\X_s/S) & \CH_k^{\T_S}(\X/S)& \CH_k^{\T_S}(\X_{\eta}/S)& 0  \\
         \CH_k(\X_s/S) & \CH_k(\X/S)& \CH_k(\X_{\eta}/S)& 0 \\
      };
      \draw[->] (m-1-1) to  (m-1-2);
       \draw[->] (m-1-2) to (m-1-3);
         \draw[->] (m-1-3) to  (m-1-4);
       \draw[->] (m-2-1) to  (m-2-2);
        \draw[->] (m-2-2) to (m-2-3);
         \draw[->] (m-2-3) to  (m-2-4);
      \draw[->] (m-1-1)-- (m-2-1);
      \draw[->] (m-1-2)-- (m-2-2);
      \draw[->] (m-1-3)-- (m-2-3);
      
     \end{tikzpicture}
     \end{center}
     \end{figure} 
The exactness of the bottom row is the localization sequence in \ref{eq:localization}. The exactness of the top row is the invariant version of the localization sequence. 

We have that $\CH^{\T_S}_k\left(\X_{\eta}/S\right) \simeq \CH^{\T_K}_{k-1}\left(\X_{\eta}/K\right)$ and $\CH^{\T_S}_k\left(\X_{s}/S\right) \simeq \CH^{\T_{\kappa}}_k\left(\X_{s}/\kappa\right)$. 
Applying \cite[Theorem 1]{fmss} to $\X_{\eta}$ and $\X_{s}$ endowed with the actions of $\T_K$ and $\T_{\kappa}$, respectively,
we obtain that $\CH^{\T_K}_{k-1}\left(\X_{\eta}/K\right) \simeq \CH_{k-1}\left(\X_{\eta}/K\right)$ and 
$\CH^{\T_{\kappa}}_k\left(\X_{s}/\kappa\right) \simeq \CH_k\left(\X_{s}/\kappa\right)$. 
Hence,  both the left and the right vertical maps of the above diagram are isomorphisms. 

Now, by \cite[Section 8]{gillet}, one has a description of the left part of the localization sequence, that is, there is an exact sequence 
\begin{eqnarray*}
\begin{tikzcd}
\CH^{n+2-k,n+1-k}(X) \arrow["{\div_{\X}}"]{r} &\CH_k(\X_s/S)\arrow["{\iota}_*"]{r} &\CH_k(\X/S)\arrow["{j}^*"]{r} &\CH_{k}(\X_{\eta}/S)\arrow{r} &0,
\end{tikzcd}
\end{eqnarray*}
where $\CH^{n+2-k,n+1-k}(\X_{\eta}) $ is a cohomology group of $X$ defined in the following way: set $p = n+2-k$. We write $X^{(p)}$ for points of codimension $p$. Then 
\[
\CH^{p,p-1}(X) \coloneqq \frac{\on{Ker}\left(d \colon \bigsqcup_{x \in X^{(p-1)}}k(x)^* \longrightarrow \bigsqcup_{x \in X^{(p)}}\Z\right)}{ \on{Im}\left(d \colon \bigsqcup_{x \in X^{(p-2)}}K_2(k(x)) \longrightarrow \bigsqcup_{x \in X^{(p-1)}}k(x)^*\right)},
\]
where $d$ is the differential that takes a rational function on a subvariety of codimension $p-1$ to its divisor, viewed as a codimension $p$ cycle, and  $K_2(k(x))$ denotes the algebraic $K$-group of Quillen (we refer to \cite{quillen} and \cite{gillet} for details). 

There is also an invariant version of this cohomology group, denoted by $\CH_{\T_K}^{p,p-1}(X)$, defined by considering distinguished points under the action of $\T_K$ of codimension $p$, that is, whose torus orbits give a stratification on $X$. Given the stratification of $X$ by torus orbits, there is a natural surjective map $\CH_{\T_K}^{p,p-1}(X) \to \CH^{p,p-1}(X)$. 

It follows that we have a diagram 

\begin{figure}[H]
\begin{center}
    \begin{tikzpicture}
      \matrix[dmatrix] (m)
      {
       \CH_{T_K}^{p,p-1}& \CH_k^{\T_S}(\X_s/S) & \CH_k^{\T_S}(\X/S)& \CH_k^{\T_S}(\X_{\eta}/S)& 0  \\
   \CH^{p,p-1}&      \CH_k(\X_s/S) & \CH_k(\X/S)& \CH_k(\X_{\eta}/S)& 0 \\
      };
      \draw[->] (m-1-1) to  (m-1-2);
       \draw[->] (m-1-2) to (m-1-3);
         \draw[->] (m-1-3) to  (m-1-4);
       \draw[->] (m-2-1) to  (m-2-2);
        \draw[->] (m-2-2) to (m-2-3);
         \draw[->] (m-2-3) to  (m-2-4);
      \draw[->] (m-1-1)-- node[right]{$\alpha_1$} (m-2-1);
      \draw[->] (m-1-2)--node[right]{$\alpha_2$} (m-2-2);
      \draw[->] (m-1-3)--node[right]{$\alpha_3$} (m-2-3);
      \draw[->] (m-1-4) -- (m-1-5);
      \draw[->] (m-2-4) -- (m-2-5);
      \draw[->] (m-1-4) --node[right]{$\alpha_4$} (m-2-4);
      
     \end{tikzpicture}
     \end{center}
     \end{figure} 
with exact rows, $\alpha_1$ an epimorphism and $\alpha_2$ and $\alpha_4$ isomorphisms. By the 5-Lemma, we get that 
$\alpha_3$ is also an isomorphism. This proves the claim.

\end{proof}

\subsection{Chow groups of the special fiber}\label{subsec:chow-special}

As before, $\X= \X_{\Pi}$ denotes a complete regular toric $S$-scheme corresponding to a complete, smooth SCR polyhedral complex $\Pi$ in $N_{R}$. We set $\Sigma = \rec(\Pi)$.

We saw already in the proof of Theorem \ref{th:inv-cycles} that the Chow groups of the special fiber are generated by vertical torus orbit closures. We now show how to calculate them explicitly in terms of an exact sequence.  For any integer $i$ write $\X_s^{[i]}$ for the disjoint union of $(i+1)$-fold intersections of irreducible components of $\X_s$. Then by \cite[Appendix A]{BGS} the Chow group of the special fiber has the following presentation for any integer $k$.
\begin{equation}\label{eq:chow-special}
CH_k\left(\X_s^{[1]}/\kappa\right)\longrightarrow CH_k\left(\X_s^{[0]}/\kappa \right) \longrightarrow CH_k\left(\X_s/\kappa \right) \longrightarrow 0.
\end{equation}
For toric schemes this translates into 
\[
\bigoplus_{\stackrel{\Lambda \in \Pi(1)}{\Lambda \text{ bdd}}}
CH_k\left(V(\Lambda)/\kappa \right) \longrightarrow \bigoplus_{v \in \Pi(0)}CH_k\left(V(v)/\kappa \right) \longrightarrow CH_k\left(\X_s /\kappa\right) \longrightarrow 0,
\]
where the term on the left is the direct sum of all \emph{bounded} one-dimensional polyhedra $\Lambda \in \Pi(1)$. This follows from the description of the special fiber in terms of $\Pi$ and regularity: if a polyhedron in $\Pi$ contains two vertices $v_1$ and $v_2$ then $v_1$ and $v_2$ are endpoints of an edge in $\Pi$.
\begin{rem}
There is an explicit combinatorial description of the maps in the above sequence which is given in \cite{bot-ara}. For the results of the present article, we only use the description of the map for $k = 0$ and $k=1$ in Proposition \ref{prop:ch-zero} and Example \ref{ex1}, respectively.
\end{rem}
\begin{exa}\label{ex:can}
Consider $\Pi = \Sigma$ the canonical model. Then since there are no bounded polyhedra of dimension $>0$ and since $\{0\}$ is the only vertex, here one easily sees that $CH_k\left(\X_s /\kappa \right) \simeq  \CH_k\left(X_{\Sigma}/\kappa\right)$.  In particular, for $k = 0$ we get
\[
CH_0\left(\X_s /\kappa \right) \simeq \Q.
\]
The isomorphisms $CH_k\left(\X_s /\kappa \right) \simeq  \CH_k\left(X_{\Sigma}/\kappa\right)$ in the case of the canonical model already appears in the work by Maillot (see \cite[Theorem~2.5.2]{maillot}).
\end{exa}
\begin{prop}\label{prop:ch-zero}
We have that 
\[
\CH_0(\X_s/\kappa) \simeq \Q
\]
 for any choice of regular, proper toric model $\X = \X_{\Pi}$. 
\end{prop}
\begin{proof}
Since the $\CH_0$ of a complete  toric variety over a field is isomorphic to $\Q$ (as follows from e.\,g.\,\cite[Proposition 2.1]{FS}), we have the exact sequence 
\begin{eqnarray*}
\begin{tikzcd}
\bigoplus_{\stackrel{\Lambda \in \Pi(1)}{\Lambda \text{ bdd}}} \Q\arrow["{a}"]{r} &\bigoplus_{v \in \Pi(0)}\Q\arrow{r} &CH_0\left(\X_s /\kappa\right)\arrow{r} &0.
\end{tikzcd}
\end{eqnarray*}
where $a$ is given in the following way (see \cite[Example 4.1]{bot-ara}).
 Let $e_{\Lambda}$ be a generator corresponding to a bounded $\Lambda \in \Pi(1)$ and let $e_{v_{\Lambda}^1}$ and $e_{v_{\Lambda}^2}$ be the generators corresponding to the endpoints of $\Lambda$. Then
\[
a\left(e_{\Lambda}\right) = e_{v_{\Lambda}^1} -  e_{v_{\Lambda}^2}.
\]
We have to show that 
\[
CH_0\left(\X_s /\kappa \right) = \left(\bigoplus_{v \in \Pi(0)}\Q\right)/\Im(a) \simeq \Q, 
\]
that is, we have to show that the difference of any two generators $e_{v_1}$ and $e_{v_2}$ of  $\bigoplus_{v \in \Pi(0)}\Q$ lies in the image of $a$. Since the $1$-skeleton of $\Pi$ is path connected we can find a sequence of segments $E_1, \dotsc, E_r \in \Pi(1)$ connecting $v_1$ and $v_2$. Then we have that 
\[
a\left(\sum_{i =1}^re_{E_i}\right) = e_{v_1}-e_{v_2}.
\]
\end{proof}
%
\begin{rem}\label{rem:kernel}
For the proof of Proposition \ref{prop:ch-zero} one can also argue in the following more abstract way. It follows from \cite[Theorem 2.2.1]{BGS} that for any integer $k$, the Kernels of $\iota^*\iota_*\colon \CH_{n-k}(\X_s) \to \CH^{k+1}(\X_s)$ for all choices of toric models are isomorphic. Hence, since for $k=0$ we get
\[
CH_0(\X_s) = \on{Ker}(\iota^*\iota_* \colon \CH_0(\X_s) \longrightarrow \CH^{n+1}(\X_s) = 0,
\]
and since, by the Example \ref{ex:can} above, $CH_0\left(\X_{\Sigma,s}/S\right) = \Q$ for the canonical model, we get that 
$\CH_0(\X_s/S) \simeq \Q$ for any toric model $\X$ over $S$. 
\end{rem}
From Proposition \ref{prop:ch-zero} and Lemma \ref{lem:zero} we obtain the following corollary.
\begin{cor}\label{cor:special-toric}
Let $\X$ be a be a proper regular toric scheme over $S$. Then $\CH_{0}(\X/S) \simeq 0$.
\end{cor}
\subsection{Examples}\label{subsec:exa}
As an application of Theorem \ref{th:exact} we calculate the Chow groups of some toric schemes.

\begin{exa}\label{exa:p1}
Consider toric models $\X= \X_{\Pi}$ of $\P^1$ with $\Pi$ and $\Sigma = \on{rec}(\Pi)$ as in the following figure. 
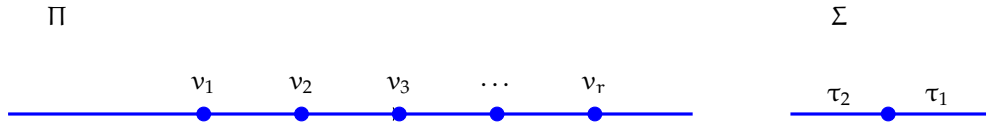
\begin{figure}[H]
\centering
\hfill

	\begin{tikzpicture}[scale = 1.3]
	\draw (-1.5,1) node{$\Pi$};
	\draw (6.5,1) node{$\Sigma$};
		\draw [->,thin] (-2,0) -- (2,0);
		\draw (0,0.3) node{$v_1$};
		\draw (1,0.3) node{$v_2$};
		\draw (2,0.3) node{$v_3$};
		\draw (3,0.3) node{$\dotsc$};
		\draw (4,0.3) node{$v_r$};
		\draw[-,blue, very thick] (-2,0) -- (5,0);
		\filldraw [blue] (0,0) circle (2pt);
		\filldraw [blue] (1,0) circle (2pt);
		\filldraw [blue] (2,0) circle (2pt);
		\filldraw [blue] (3,0) circle (2pt);
		\filldraw [blue] (4,0) circle (2pt);
		\draw[-,blue, very thick] (6,0) -- (8,0);
		\filldraw [blue] (7,0) circle (2pt);
		\draw (6.5,0) node[above]{$\tau_2$};
		\draw (7.5,0) node[above]{$\tau_1$};

	\end{tikzpicture}
\end{figure}

 We have $r = |\Pi(0)|$. We claim that
\[
\CH_k(\X/S) \simeq \begin{cases}  
\Q^r \; &\text{ if } \; k= 1, \\
\Q \; &\text{ if } \; k =2,\\
0 \; &\text{ otherwise.}
\end{cases}
\]
Moreover, for $k=1$ the relations between the one-cycles in $\CH_1(\X/S)$ are given by 
\[
\sum_{i = 1}^r V(v_i) = 0 \; \text{ and }\; V(\tau_1) - V(\tau_2) =- \sum_{i = 1}^{r-1}v_iV(v_{i +1}),
\]
 where $V(v_i)$ denotes the vertical cycle corresponding to the vertex $v_i \in \Pi(0)$ and where $V(\tau_i)$ denotes the horizontal cycle corresponding to the ray $\tau_i \in \Sigma(1)$. 
 
Hence a set of generators is given by $V(v_i)$ for $i = 1, \dotsc, r-1$ and by $V(\tau_1)$.

In order to prove the claim, consider first the case $k = 0$. We have an exact sequence
\[
\Q^{r+2} \longrightarrow \Q^{r+1} \longrightarrow \CH_{0}(\X/S) \longrightarrow 0,
\]
where the first map is given by 
\[
\left(m_1, \dotsc, m_{r+2}\right) \longmapsto \left(m_1-m_2, m_2-m_3, 
\dotsc, m_{r+1} - m_{r+2}\right).
\]
One sees that this map is surjective and hence $\CH_{0}(\X/S) \simeq 0$. This is a particular case of Corollary \ref{cor:special-toric}.  

Now we compute $\CH_1(\X/S)$. In this case we have the exact sequence 
\[
M_{\Q} \oplus \Q \longrightarrow \Q^{r+2} \longrightarrow \CH_1(\X/S) \longrightarrow 0
\]
where the first map is given by 
\[
(m,\ell) \longmapsto \left(m, -m, mv_1+\ell,mv_2+\ell, \dotsc, mv_r+\ell \right).
\]
This has image $\Q^2$ which implies $\CH_1(\X/S) \simeq \Q^r$.

The case $k=2$ follows from Example \ref{ex:ch-max}.

We can also compute the Chow groups of the special fiber using the combinatorial from of the exact sequence \eqref{eq:chow-special}. We claim that
\[
\CH_k(\X_s/S)\simeq\CH_k(\X_s/\kappa) \simeq \begin{cases} \Q \; &\text{ if } \; k= 0, \\ 
\Q^r \; &\text{ if } \; k= 1,\\
0 \; &\text{otherwise.}
\end{cases}
\] 
Note that all the cases $k \neq 0$ are clear and for $k=0$ this is a particular case of Proposition \ref{prop:ch-zero}. We can also calculate this explicitly. Note that the number of bounded rays is $ r-1$. Then the exact sequence \eqref{eq:chow-special} has the form
\[
\Q^{r-1} \longrightarrow \Q^r \longrightarrow CH_0\left(\X_s /\kappa\right) \longrightarrow 0,
\]
where the first map is given by 
\[
\left(m_1, \dotsc, m_{r-1}\right) \longmapsto \left( - m_1, m_1-m_2, \dotsc, m_{r-2}-m_{r-1},m_{r-1}\right).
\]
This map is injective and hence $CH_0(\X_s/\kappa) \simeq \Q$. 
\end{exa}

\begin{exa}\label{ex1}
We now consider the toric model $\X= \X_{\Pi}$ of $\P^2$ given by the following polyhedral complex 
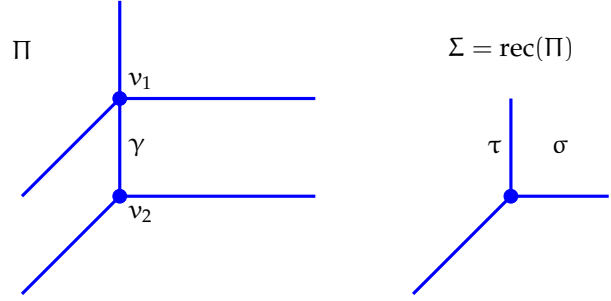
\begin{figure}[H]
\centering
\hfill

	\begin{tikzpicture}[scale = 1.3]
\draw[-,blue, very thick] (0,0) -- (2,0);
\draw[-,blue, very thick] (0,0) -- (0,2);
\draw[-,blue, very thick] (0,1) -- (2,1);
\draw[-,blue, very thick] (0,0) -- (-1,-1);
\draw[-,blue, very thick] (0,1) -- (-1,0);
\filldraw [blue] (0,0) circle (2pt);
		\draw (0,0) node[below right]{$v_2$};
		\filldraw [blue] (0,1) circle (2pt);
		\draw (0,1) node[above right]{$v_1$};
		\draw (0,0.5) node[right]{$\gamma$};
		\draw (-1,1.5) node{$\Pi$};
		
		\draw[-,blue, very thick] (4,0) -- (5,0);
\draw[-,blue, very thick] (4,0) -- (4,1);
\draw[-,blue, very thick] (4,0) -- (3,-1);
\filldraw [blue] (4,0) circle (2pt);
\draw (4.5,0.5) node{$\sigma$};
\draw (4,0.5) node[left]{$\tau$};
		
		\draw (4,1.5) node{$\Sigma = \on{rec}(\Pi)$};

	\end{tikzpicture}
\end{figure}

We claim that 
\[
\CH_k(\X/S) \simeq \begin{cases}  
\Q^2 \; &\text{ if } \; k= 1,2 \\
\Q \; &\text{ if } \; k =3,\\
0 \; &\text{ otherwise.}
\end{cases}
\]
Moreover, we can give a set of generators. For $k = 1$ a set of generators is given by $V(\sigma)$, the horizontal cycle corresponding to the cone $\sigma \in \Sigma(2)$ appearing in the above Figure and the vertical cycle $V(\gamma)$ corresponding to $\gamma$.
For $k = 2$ a set of generators is given by $\{V(\tau), V(v_2)\}$, where $V(\tau)$ is the horizontal cycle corresponding to the cone $\tau \in \Sigma(1)$ and $V(v_2)$ is the irreducible component of the special fiber corresponding to the vertex $v_1 \in \Pi(0)$.

For the special fiber we claim that
\[
\CH_k(\X_s/\kappa) \simeq \begin{cases}  
\Q \; &\text{ if } \; k= 0 \\
\Q^2 \; &\text{ if } \; k =1,2\\
0 \; &\text{ otherwise.}
\end{cases}
\]
In order to prove the claim for $\CH_k(\X/S)$, it suffices to compute $\CH_k(\X/S)$ for $k=1,2$ since all other cases are clear from Corollary \ref{cor:special-toric} and Example \ref{ex:ch-max}. For $k =1$ we use the exact sequence in Theorem \ref{th:exact}. This has the form 
\[
\Q^{10} \xrightarrow{P} \Q^9 \to \CH_1(\X/S) \to 0,
\]
where, after an appropriate choice of generators, $P$ is the linear map given by 
\setcounter{MaxMatrixCols}{15}
\[
P = \begin{bmatrix} 1&0&-1&0&0&0&0&0&0&0\\ 0&0&1&0&-1&0&0&0&0&0\\-1&0&0&0&1&0&0&0&0&0\\0&0&0&0&0&0&0&-1&0&1\\0&0&0&1&0&0&1&0&0&0\\0&0&0&1&0&0&0&0&1&0\\0&0&0&0&0&1&0&0&-1&-1\\0&1&0&0&0&0&0&1&0&0 \\
0&0&0&0&0&0&0&0&0&0 \end{bmatrix}
\]
We have $\on{rank}(P) = 7$ and hence $\CH_1(\X/S) \simeq \Q^2$.

For $k=2$ we have
\[
\Q^3 \longrightarrow \Q^5 \longrightarrow \CH_2(\X/S) \longrightarrow 0
\]
where the first map is given by $(m_1,m_2,\ell) \mapsto (m_1, \langle m, v_1\rangle + \ell, \langle m, v_2\rangle + \ell)$, where $m = (m_1, m_2)$. 
This map is injective hence $\CH_2(\X/S) \simeq \Q^2$. 

The claim about the generators follows from the above computations.  

To prove the claim for the special fiber $\CH(\X_s/\kappa)$ we only need to consider the cases $k=0,1$ since the other cases are clear.

For $k=0$, from \eqref{eq:chow-special} we get an exact sequence 
\[
\Q \longrightarrow \Q^2 \longrightarrow \CH_0(\X_s/\kappa) \longrightarrow 0,
\]
where the first map is of the form $z \mapsto (z,-z)$. This map is injective and hence we get $\CH_0(\X_s/\kappa) \simeq \Q$, as also follows from Proposition \ref{prop:ch-zero}.  
For $k=1$, the exact sequence \eqref{eq:chow-special} is of the form 
\[
\Q \longrightarrow \Q^2 \oplus \Q \longrightarrow \CH_1(\X_s/\kappa) \longrightarrow 0.
\]
Indeed, in the second term, the $\Q^2$ corresponds to $\CH_1(V(v_1)/\kappa)$ which is the class group of the blow up of $\mathbb{P}^2$, and the $\Q$-term corresponds to $\CH_1\left(V(v_2)/\kappa\right)$ which is the class group of $\mathbb{P}^2$. The first map is then given by $q \mapsto (q,0,q)$ (which can be deduced from \cite[Example~4.1]{bot-ara}). Therefore we get $\CH_1(\X_s/\kappa) \simeq \Q^3/\Q \simeq \Q^2$. 

\end{exa}
\begin{exa}\label{ex2}
We now consider a toric model of $\P^2_{\on{Bl}}$, a blow up of $\P^2$ at a torus fixed point. Let $\X = \X_{\Pi}$ be the toric scheme corresponding to the polyhedral complex $\Pi$ in the following figure. 
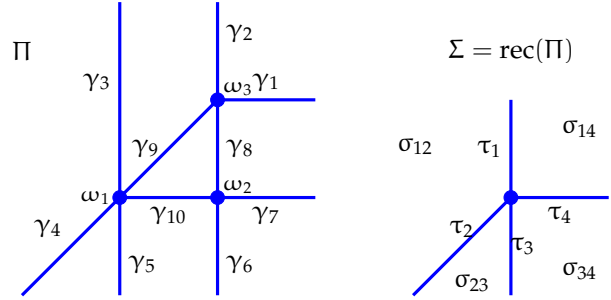
\begin{figure}[H]
\centering
\hfill

	\begin{tikzpicture}[scale = 1.3]
\draw[-,blue, very thick] (0,0) -- (2,0);
\draw[-,blue, very thick] (0,0) -- (0,2);
\draw[-,blue, very thick] (1,1) -- (2,1);
\draw[-,blue, very thick] (0,0) -- (1,1);
\draw[-,blue, very thick] (0,0) -- (-1,-1);
\draw[-,blue, very thick] (1,0) -- (1,-1);
\draw[-,blue, very thick] (1,1) -- (1,2);
\draw[-,blue, very thick] (1,0) -- (1,1);
\draw[-,blue, very thick] (0,0) -- (0,-1);

\filldraw [blue] (0,0) circle (2pt);
		\draw (-0.5,-0.5) node[above left]{$\gamma_4$};
		\draw (1,-0.5) node[below right]{$\gamma_6$};
		\draw (0,-0.5) node[below right]{$\gamma_5$};
		\filldraw [blue] (1,0) circle (2pt);
	
		\filldraw [blue] (1,1) circle (2pt);
		
		\draw (0.5,0) node[below]{$\gamma_{10}$};
		\draw (1,0.5) node[right]{$\gamma_8$};
		\draw (0.5,0.5) node[left]{$\gamma_9$};
		\draw (1.5,1) node[above]{$\gamma_1$};
		\draw (1,1.5) node[above right]{$\gamma_2$};
		\draw (0,1) node[above left]{$\gamma_3$};
		\draw (1.5,0) node[below]{$\gamma_7$};
		\draw (0,0) node[left]{\footnotesize{$\omega_1$}};
		\draw (0.95,-0.05) node[above right]{\footnotesize{$\omega_2$}};
		\draw (0.95,0.95) node[above right]{\footnotesize{$\omega_3$}};

		\draw (-1,1.5) node{$\Pi$};
		
		\draw[-,blue, very thick] (4,0) -- (5,0);
\draw[-,blue, very thick] (4,0) -- (4,1);
\draw[-,blue, very thick] (4,0) -- (3,-1);
\draw[-,blue, very thick] (4,0) -- (4,-1);
\filldraw [blue] (4,0) circle (2pt);
		\draw (4.7,0.7) node{$\sigma_{14}$};
\draw (4.5,0) node[below]{$\tau_4$};
\draw (3.5,-0.5) node[above]{$\tau_2$};
\draw (4,0.5) node[left]{$\tau_1$};
\draw (3.9,-0.5) node[right]{$\tau_3$};
\draw (3.3,0.7) node[below left]{$\sigma_{12}$};
\draw (4.7,-0.75) node{$\sigma_{34}$};
\draw (3.6,-0.7) node[below]{$\sigma_{23}$};
		\draw (4,1.5) node{$\Sigma = \on{rec}(\Pi)$};

	\end{tikzpicture}
\end{figure}
We claim that 
\[
\CH_k(\X/S) \simeq \begin{cases}  
\Q \; &\text{ if } \; k= 3 \\
\Q^4 \; &\text{ if } \; k =2,\\
\Q^3 \; &\text{ if }\; k = 1, \\
0 \; &\text{ otherwise.}
\end{cases}
\]
Moreover, for $k=1$, generators are given by the closures of the torus orbits corresponding to $\sigma_{14}$, $\gamma_8$ and $\gamma_{10}$. 

For $k=2$  a set of generators is given by the closures of the torus orbits corresponding to $\omega_3, \tau_2, \tau_3$ and $\tau_4$.

In order to prove the claim note that all the cases except $k = 1,2$ are clear. We use the exact sequence in Theorem \ref{th:exact}. 
We specify that $\omega_1 = (0,0)$, $\omega_2 = (1,0)$ and $\omega_3 = (1,1)$ for the following calculations. 

For $k= 1$ we have
\begin{eqnarray*}
\begin{tikzcd}
\Q^8 \oplus \Q^6\arrow["{M}"]{r} &\Q^4 \oplus \Q^{10}\arrow{r} &\CH_1(\X/S) \arrow{r} &0,
\end{tikzcd}
\end{eqnarray*}
where, with an appropriate choice of generators, $M$ is given by the linear map 
\setcounter{MaxMatrixCols}{14}
\[
M = \begin{bmatrix} 1&0&0&0&0&0&-1&0&0&0&0&0&0&0\\ 
-1&0&1&0&0&0&0&0&0&0&0&0&0&0\\
0&0&-1&0&1&0&0&0&0&0&0&0&0&0\\
0&0&0&0&-1&0&1&0&0&0&0&0&0&0\\
0&0&0&0&0&0&1&1&1&0&0&0&0&0\\
1&1&0&0&0&0&0&0&0&1&0&0&0&0\\
1&1&0&0&0&0&0&0&0&0&0&1&0&0\\
0&0&1&1&0&0&0&0&0&0&-1&-1&0&0\\
0&0&0&0&1&1&0&0&0&0&0&-1&0&0\\
0&0&0&0&1&1&0&0&0&0&0&0&0&-1\\
0&0&0&0&0&0&1&1&0&0&0&0&1&0\\
0&0&0&0&0&0&0&0&0&-1&0&0&0&1\\
0&0&0&0&0&0&0&0&-1&-1&1&1&0&0\\
0&0&0&0&0&0&0&0&0&0&1&0&-1&0
 \end{bmatrix}
\]
We compute $\on{rank}(M) = 11$. Hence we obtain $\CH_1(\X/S) \simeq \Q^3$. In order to prove the claim regarding the generators, note that from $M$ we get all the relations.  We get 
\[
\sigma_{12}, \sigma_{23}, \sigma_{34}\sim -\sigma_{14}, \; \gamma_4 \sim 0, \; \gamma_5\sim -\gamma_8, \gamma_6 \sim \gamma_8, \; \gamma_7 \sim \gamma_10
\]
and 
\[
\gamma_9\sim -\gamma_{10}, \; \gamma_1 \sim -\gamma_{10},\; \gamma_2 \sim \gamma_8-\gamma_{10},\; \gamma_3 \sim \gamma_{10}-\gamma_8.
\]
Here, given a polyhedron $\mathfrak{p}$ we have written also $\mathfrak{p}$ to denote the corresponding subvariety.
This proves the claim for $k=1$.
For $k=2$ the exact sequence is of the form 
\begin{eqnarray*}
\begin{tikzcd}
\Q^3\arrow["{A}"]{r} &\Q^4 \oplus \Q^{3}\arrow{r} &\CH_2(\X/S) \arrow{r} &0,
\end{tikzcd}
\end{eqnarray*}

where, with an appropriate choice of generators, $A$ is given by the linear map
\setcounter{MaxMatrixCols}{3}
\[
A = \begin{bmatrix} 0&1&0\\
-1&-1&0\\
0&-1&0\\
1&0&0\\
0&0&1\\
1&0&1\\
1&1&1
 \end{bmatrix}
\]
We compute $\on{rank}(A) = 3$ and hence $\CH_2(\X/S) \simeq \Q^4$. To see the statement regarding the generators, note that we have the relations
\[
\tau_1 \sim \tau_2+\tau_3-\omega_3, \; \omega_2 \sim \tau_2-\tau_4-\omega_3,\; \omega_1\sim -\omega_3-\omega_2 \sim \tau_4-\tau_2,
\]
where, as above, given a polyhedron $\mathfrak{p}$ we have written also $\mathfrak{p}$ to denote the corresponding subvariety.

\end{exa}

\subsection{A formula for the rank}\label{subsec:rank}

Let $\Pi$ be a complete smooth SCR polyhedral complex in $N_{\R}$ and let $\X = \X_{\Pi}$ the corresponding complete, regular toric scheme of $S$-absolute dimension $n+1$. 

Consider the fan $c(\Pi)$ in $N_{\R} \times \R_{\geq 0}$. Since $\X$ is constructed from the cones of $c(\Pi)$ in basically the same way that a toric variety over a field is constructed from a fan, in some sense we may think of $\X$ as the (non-complete) toric variety $X_{c(\Pi)}$ with corresponding fan $c(\Pi)$. Using Poincaré dulaity (\cite[Theorem 3.16]{bot-ara}) and the description of the Chow cohomology groups of toric models in terms of piecewise polynomial functions (\cite[Theorem 3.19]{bot-ara}), in the same way as in the proof of \cite[Corollary~3.1]{BR}, one obtains the following formula for the rank of the Chow groups.

\begin{eqnarray}\label{eq:rank}
\sum_{k=0}^{n+1}\dim\CH_{n+1-k}(\X/S)z^k = \sum_{\sigma \in c(\Pi)}z^{\dim(\sigma)}(1-z)^{\codim(\sigma)}.
\end{eqnarray}
 
Moreover, from Corollary \ref{cor:special-toric} we know that $\CH_0(\X/S) = 0$, hence we only need to consider the sum up to $n$. 
We now verify the above formula for our examples in Section \ref{subsec:exa}.
\begin{exa}
Consider the toric model of $\P^1$ as in Example \ref{exa:p1}. Here the right hand side of \eqref{eq:rank} reads
\[
(r+1)z^2+ (r+2)z(1-z) + (1-z)^2 = rz+1.
\]
Hence we get 
\[
\dim\CH_k(\X/S) = \begin{cases} 0 \; &\text{ if }k = 0\\ r \; &\text{ if }k = 1\\ 1 \; &\text{ if }k = 2.\end{cases}.
\]
This agrees with the computations in Example \ref{exa:p1}.

For the toric model from Exampele \ref{ex1}, the right hand side of \eqref{eq:rank} is 
\[
5z^3 + 9z^2(1-z)+5z(1-z)^2+(1-z)^3 =  2z^2+2z+1
\]
and we see that the computations agree. 

Finally, for the toric model from Example \ref{ex2}, the right hand side of \eqref{eq:rank} is 
\[
8z^3 + 14z^2(1-z)+7z(1-z)^2+(1-z)^3 = 1 + 4z + 3z^2,
\]
so the computations agree. 
\end{exa}


\printbibliography

\end{document}